\documentclass[11pt,reqno]{amsart}

\usepackage{amsmath,amssymb,amsthm,mathtools}
\usepackage[margin=1in]{geometry}
\usepackage{microtype}
\usepackage{enumitem}
\usepackage[hidelinks]{hyperref}

\newtheorem{theorem}{Theorem}[section]
\newtheorem{lemma}[theorem]{Lemma}
\newtheorem{proposition}[theorem]{Proposition}

\theoremstyle{definition}

\numberwithin{equation}{section}

\newcommand{\rhostar}{\rho^{*}}
\newcommand{\sigmastar}{\sigma^{*}}

\title[A Huppert-type $\rho^*$--$\sigma^*$ theorem]{A Huppert-type $\rho^*$--$\sigma^*$ theorem for sub-class sizes}
\author{Yong Yang}
\address{Department of Mathematics, Texas State University, 601 University Drive, San Marcos, TX 78666, USA}
\email{yang@txstate.edu}
\date{}

\begin{document}

\begin{abstract}
For a finite group $G$, let $\rho^{*}(G)$ be the set of primes dividing some sub-class size of $G$, and let $\sigma^{*}(G)$ be the maximum number of distinct prime divisors of a sub-class size. Motivated by the conjugacy-class version of Huppert's $\rho$--$\sigma$ problem, we prove that
\[
 |\rho^{*}(G)|<2\sigma^{*}(G)
\]
for every finite nonnilpotent group $G$. This answers \cite[Remark~5.3]{QY} and extends to arbitrary finite nonnilpotent groups the coefficient-$2$ bound proved in \cite[Theorem~1.7]{QY} for solvable nonnilpotent groups. The coefficient $2$ is best possible. The proof does not use the classification of finite simple groups.
\end{abstract}

\maketitle

\section{Introduction}

Huppert's $\rho$--$\sigma$ problem \cite{Huppert} asks whether the total number of primes occurring in conjugacy class sizes is at most twice the largest number occurring in a single class size. The coefficient $2$ is not valid even for solvable groups; see \cite{CD} and the discussion in \cite[Section~1]{QY}.

Let $G$ be a finite group and $x\in G$. Denote by $G_x$ the smallest subnormal subgroup of $G$ containing $x$. Following \cite{QY}, the \emph{sub-class} of $x$ in $G$ is
\[
 x^{G^{*}}:=x^{G_x},
\]
and its size $|x^{G^{*}}|$ is called a \emph{sub-class size}. Set
\[
 \rhostar(G)=\bigcup_{x\in G}\pi\bigl(|x^{G^{*}}|\bigr),
 \qquad
 \sigmastar(G)=\max_{x\in G}\bigl|\pi\bigl(|x^{G^{*}}|\bigr)\bigr|,
\]
where $\pi(n)$ denotes the set of prime divisors of the positive integer $n$.

For sub-class sizes, a stronger bound holds. It was proved in \cite[Theorem~1.7]{QY} that if $G$ is solvable and nonnilpotent, then
\[
 |\rhostar(G)|<2\sigmastar(G).
\]
Thus, for solvable groups, sub-class sizes satisfy the coefficient-$2$ bound that fails for ordinary conjugacy class sizes. It was asked in \cite[Remark~5.3]{QY} whether a corresponding result holds for arbitrary groups. We show that the same strict inequality holds for every finite nonnilpotent group. Using the family in \cite[Example~1]{CD}, we also show that the coefficient $2$ is best possible. No classification of finite simple groups is needed. For nilpotent groups all sub-class sizes are $1$ \cite[Corollary~2.3]{QY}, so $\rhostar(G)=\varnothing$ and $\sigmastar(G)=0$.

\begin{theorem}\label{thm:main}
Let $G$ be a finite nonnilpotent group. Then
\[
 |\rhostar(G)|<2\sigmastar(G).
\]
\end{theorem}

Casolo \cite{Casolo} proved the corresponding class-size estimate by fixing a prime $p$, counting the elements whose conjugacy class size is divisible by $p$, and then double-counting. We use the same idea for sub-classes. For a prime $p$, define
\[
 \Delta_p(G)=\{x\in G: p\mid |x^{G^{*}}|\}.
\]
For $p$-solvable groups, \cite[Lemma~3.1]{QY} gives the required estimate. We prove it for every finite group.

\begin{theorem}\label{thm:density}
Let $G$ be a finite group and let $p$ be a prime. If $G$ is not $p$-nilpotent, then
\[
 |\Delta_p(G)|\ge \frac{|G|}{2}.
\]
\end{theorem}

Theorem~\ref{thm:main} then follows from the counting argument in \cite[Theorem~1.7]{QY}.

\section{Preliminaries}

We need the following facts about sub-classes and components.

\begin{lemma}\label{lem:quotient}
Let $N\lhd G$ and $x\in G$. Then
\[
 |(xN)^{(G/N)^{*}}|\mid |x^{G^{*}}|.
\]
\end{lemma}

\begin{proof}
This is the quotient divisibility property for sub-class sizes; see \cite[Lemma~2.1]{QY}.
\end{proof}

We use only the elementary direction of the $p$-nilpotence criterion for sub-class sizes.

\begin{lemma}\label{lem:pnil}
Let $G$ be a finite group and let $p$ be a prime. If $G$ is $p$-nilpotent, then
\[
 p\nmid |x^{G^{*}}|
 \qquad\text{for every }x\in G.
\]
Equivalently, if $p\in\rhostar(G)$, then $G$ is not $p$-nilpotent.
\end{lemma}

\begin{proof}
This is the implication $(3)\Rightarrow(1)$ in \cite[Theorem~1.4]{QY}; its proof is elementary. We recall it briefly. Write $x=uv=vu$, where $u$ is a $p$-element and $v$ is a $p'$-element, and let $O_{p'}(G)$ be the normal $p$-complement. Set
\[
 K=\langle u\rangle O_{p'}(G).
\]
Since $G/O_{p'}(G)$ is a $p$-group, the subgroup $K/O_{p'}(G)$ is subnormal in $G/O_{p'}(G)$; taking inverse images gives $K\lhd\lhd G$. Also $x\in K$. Since $\langle u\rangle$ is a Sylow $p$-subgroup of $K$ and $u$ commutes with $v$, we have $\langle u\rangle\le C_K(x)$. Thus
\[
 p\nmid |x^K|=[K:C_K(x)].
\]
Because $K$ is subnormal in $G$ and contains $x$, the sub-class of $x$ is unchanged when computed in $K$; see \cite[Remark~1.8]{QY}. Hence
\[
 |x^{G^{*}}|=|x^{K^{*}}|=|x^{K_x}|.
\]
Since $K_x\lhd\lhd K$, ordinary conjugacy-class sizes divide along a subnormal chain. Indeed, if $A\lhd B$ and $x\in A$, then
\[
 [A:C_A(x)]=[AC_B(x):C_B(x)]\mid [B:C_B(x)],
\]
and iteration along a subnormal chain gives
\[
 |x^{K_x}|\mid |x^K|.
\]
Consequently $p\nmid |x^{G^{*}}|$.
\end{proof}

We also need the following standard property of components. Recall that a component is a subnormal quasisimple subgroup.

\begin{lemma}\label{lem:component}
Let $L$ be a component of a finite group $G$ and let $H\lhd\lhd G$. Then either
\[
 [L,H]=1
\]
or
\[
 L\le H.
\]
\end{lemma}

\begin{proof}
This is the standard component property; see \cite[31.4]{Aschbacher}.
\end{proof}

We also use the following fact.

\begin{lemma}\label{lem:pnormal-subgroups}
Every subgroup of a $p$-nilpotent finite group is $p$-nilpotent.
\end{lemma}

\begin{proof}
Let $X$ be $p$-nilpotent and let $K\lhd X$ be a normal $p$-complement. If $Y\le X$, then $Y\cap K\lhd Y$, the subgroup $Y\cap K$ is a $p'$-group, and
\[
 Y/(Y\cap K)\cong YK/K
\]
is a $p$-group. Thus $Y\cap K$ is a normal $p$-complement of $Y$.
\end{proof}

\section{The density theorem}

\begin{proof}[Proof of Theorem~\ref{thm:density}]
Suppose the theorem is false, and choose a counterexample $G$ of minimal order. Thus $G$ is not $p$-nilpotent and
\begin{equation}\label{eq:counter}
 |\Delta_p(G)|<\frac{|G|}{2}.
\end{equation}

\medskip
\noindent\textbf{Step 1: every proper nontrivial quotient is $p$-nilpotent.}
Let $1\ne N\lhd G$. Suppose that $G/N$ is not $p$-nilpotent. By minimality,
\[
 |\Delta_p(G/N)|\ge \frac{|G/N|}{2}.
\]
By Lemma~\ref{lem:quotient}, if $xN\in \Delta_p(G/N)$ then every lift $x$ belongs to $\Delta_p(G)$. Hence
\[
 |\Delta_p(G)|\ge |N|\,|\Delta_p(G/N)|\ge \frac{|G|}{2},
\]
contrary to \eqref{eq:counter}. Therefore
\begin{equation}\label{eq:quotients}
 G/N\text{ is $p$-nilpotent for every }1\ne N\lhd G.
\end{equation}

\medskip
\noindent\textbf{Step 2: $G$ has a unique minimal normal subgroup.}
Suppose that $N_1$ and $N_2$ are distinct minimal normal subgroups of $G$. Then $N_1\cap N_2=1$, and the diagonal map embeds $G$ into
\[
 G/N_1\times G/N_2.
\]
By \eqref{eq:quotients}, both factors are $p$-nilpotent, hence so is their direct product. Lemma~\ref{lem:pnormal-subgroups} then implies that $G$ is $p$-nilpotent, a contradiction. Thus $G$ has a unique minimal normal subgroup $N$.

We first prove the result when the group is $p$-solvable. Let $X$ be a $p$-solvable finite group, and let $N_p(X)$ denote its unique maximal normal $p$-nilpotent subgroup. We claim that
\begin{equation}\label{eq:psolvable-exceptional}
 X\setminus\Delta_p(X)=N_p(X).
\end{equation}
Put $H=X_x$. We have
\[
 H\text{ is $p$-nilpotent}\quad\Longleftrightarrow\quad x\in N_p(X).
\]
Indeed, if $H$ is $p$-nilpotent, then the normal closure $\langle H^g:g\in X\rangle$ is $p$-nilpotent by the elementary normal-closure lemma \cite[Lemma~2.2]{QY}, so $x\in N_p(X)$. Conversely, if $x\in N_p(X)$, then $H\le N_p(X)$; hence $H$ is $p$-nilpotent by Lemma~\ref{lem:pnormal-subgroups}. In that case Lemma~\ref{lem:pnil}, applied to $H$ (for which the smallest subnormal subgroup containing $x$ is $H$ itself), gives
\[
 p\nmid [H:C_H(x)]=|x^{X^{*}}|.
\]

For the converse, suppose that $p\nmid |x^{X^{*}}|$ but $x\notin N_p(X)$. Then $H$ is not $p$-nilpotent. Set
\[
 L=O_{p'}(H),\qquad M/L=O_p(H/L).
\]
Then $M$ is the maximal normal $p$-nilpotent subgroup of $H$. Indeed, if $R\lhd H$ is $p$-nilpotent, its normal $p$-complement is necessarily the largest normal $p'$-subgroup $O_{p'}(R)$. This subgroup is characteristic in $R$, hence normal in $H$, and so $O_{p'}(R)\le L$. Thus $RL/L$ is a normal $p$-subgroup of $H/L$, whence $RL/L\le M/L$ and $R\le M$. Since $H$ itself is not $p$-nilpotent, $M<H$. Since $x$ lies in no proper normal subgroup of $H$ (otherwise the minimality of $H=X_x$ would be contradicted), we have $x\notin M$. We use the standard $p$-solvable centralizer theorem: for every finite $p$-solvable group $Y$,
\[
 C_{Y/O_{p'}(Y)}\bigl(O_p(Y/O_{p'}(Y))\bigr)
 \le O_p(Y/O_{p'}(Y));
\]
see \cite[6.4.3]{KS}. Applied to $Y=H$, this gives
\[
 C_{H/L}(M/L)\le M/L.
\]
Hence $xL$ does not centralize $M/L$. Therefore
\[
 p\mid [M/L:C_{M/L}(xL)],
\]
and this index divides $[H/L:C_{H/L}(xL)]$ by normal-subgroup class-size divisibility; the latter index divides $[H:C_H(x)]$ by quotient class-size divisibility. Consequently
\[
 p\mid [H:C_H(x)]=|x^{X^{*}}|,
\]
a contradiction. This proves \eqref{eq:psolvable-exceptional}; compare \cite[Lemma~3.1]{QY}.

If $X$ is not $p$-nilpotent, then $N_p(X)<X$, so $|N_p(X)|\le |X|/2$. By \eqref{eq:psolvable-exceptional},
\[
 |\Delta_p(X)|=|X|-|N_p(X)|\ge |X|/2.
\]
Therefore the present counterexample is not $p$-solvable.

By \eqref{eq:quotients}, $G/N$ is $p$-nilpotent and therefore $p$-solvable. If $N$ were abelian, then $N$ would also be $p$-solvable, and closure of $p$-solvability under extensions would make $G$ $p$-solvable, a contradiction. Hence $N$ is nonabelian. Consequently
\[
 N=S_1\times\cdots\times S_t\cong S^t
\]
for a nonabelian simple group $S$. Moreover, $p\mid |S|$; otherwise $N$ would be a $p'$-group, hence $p$-solvable, and again extension-closure would make $G$ $p$-solvable.

We also have
\begin{equation}\label{eq:centralizerN}
 C_G(N)=1.
\end{equation}
Indeed, $C_G(N)\lhd G$ and
\[
 C_G(N)\cap N=Z(N)=1.
\]
If $C_G(N)\ne1$, then the uniqueness of the minimal normal subgroup forces $N\le C_G(N)$, contradicting $C_G(N)\cap N=1$.

\medskip
\noindent\textbf{Step 3: an element outside $\Delta_p(G)$ fixes every component of $N$.}
Let
\[
 E_p(G)=G\setminus \Delta_p(G),
\]
and take $x\in E_p(G)$. Put $H=G_x$. Since
\[
 p\nmid |x^{G^{*}}|=[H:C_H(x)],
\]
there exists $Q\in\operatorname{Syl}_p(H)$ such that
\begin{equation}\label{eq:Qcentralizes}
 Q\le C_H(x).
\end{equation}

Suppose that $x$ moves a component $S_i$ of $N$, say
\[
 S_i^x=S_j\ne S_i.
\]
Since $x\in H$, we have $[S_i,H]\ne1$. The subgroup $S_i$ is a component of $G$, so Lemma~\ref{lem:component} gives $S_i\le H$. Hence also $S_j=S_i^x\le H$.

Set $M=N\cap H$. Since $N\lhd G$, we have $M\lhd H$. Therefore
\[
 Q\cap M\in\operatorname{Syl}_p(M).
\]
Since $S_i\le M$ and $S_i\lhd M$,
\[
 Q\cap S_i\in\operatorname{Syl}_p(S_i).
\]
As $p\mid |S_i|$, choose $1\ne u\in Q\cap S_i$. By \eqref{eq:Qcentralizes}, $u^x=u$. On the other hand,
\[
 u^x\in S_i^x=S_j,
\]
so
\[
 u\in S_i\cap S_j=1,
\]
a contradiction. Therefore every element of $E_p(G)$ fixes each component $S_i$ setwise.

Let
\[
 K=\ker\bigl(G\to\operatorname{Sym}\{S_1,\ldots,S_t\}\bigr).
\]
Then
\begin{equation}\label{eq:EinsideK}
 E_p(G)\subseteq K.
\end{equation}
The product of the components in any $G$-orbit is normal in $G$; since $N$ is minimal normal, there is only one orbit. Thus $G$ acts transitively on the set of components of $N$. If $t>1$, then $|G:K|\ge2$, and \eqref{eq:EinsideK} gives
\[
 |E_p(G)|\le |K|\le \frac{|G|}{2}.
\]
Thus $|\Delta_p(G)|\ge |G|/2$, contrary to \eqref{eq:counter}. Hence
\[
 t=1.
\]

Thus $N=S$ is nonabelian simple. By \eqref{eq:centralizerN}, conjugation embeds $G$ into $\operatorname{Aut}(S)$, so
\[
 S\le G\le\operatorname{Aut}(S).
\]

\medskip
\noindent\textbf{Step 4: the almost simple case.}
Let $x\in E_p(G)$. If $x=1$, then $x$ centralizes every Sylow $p$-subgroup of $S$. Assume $x\ne1$, and again put $H=G_x$. Since $C_G(S)=1$, we have $[S,x]\ne1$, and therefore $[S,H]\ne1$. By Lemma~\ref{lem:component},
\[
 S\le H.
\]
Choose $Q\in\operatorname{Syl}_p(H)$ with $Q\le C_H(x)$. Since $S\lhd H$,
\[
 P:=Q\cap S\in\operatorname{Syl}_p(S),
\]
and $x\in C_G(P)$. Consequently
\begin{equation}\label{eq:Eunion}
 E_p(G)\subseteq \bigcup_{P\in\operatorname{Syl}_p(S)} C_G(P).
\end{equation}

Fix $P\in\operatorname{Syl}_p(S)$. Since $S\lhd G$, the Frattini argument gives
\[
 G=S N_G(P),
\]
and hence
\begin{equation}\label{eq:frattini-index}
 |G:N_G(P)|=|S:N_S(P)|=|\operatorname{Syl}_p(S)|.
\end{equation}
All Sylow $p$-subgroups of $S$ are conjugate in $S$, so their centralizers in $G$ have the same order. From \eqref{eq:Eunion} and \eqref{eq:frattini-index},
\begin{align}
 |E_p(G)|
 &\le |\operatorname{Syl}_p(S)|\,|C_G(P)|  
 \label{eq:unionbound}\\ \notag
 &= |G:N_G(P)|\,|C_G(P)| \\ \notag
 &= \frac{|G|}{|N_G(P):C_G(P)|}.
\end{align}

Now
\[
 N_S(P)/C_S(P)\hookrightarrow N_G(P)/C_G(P).
\]
We claim that $N_S(P)\ne C_S(P)$. Otherwise $N_S(P)=C_S(P)$, so every element of $N_S(P)$ centralizes $P$; in particular $P\le Z(N_S(P))$. Burnside's normal $p$-complement theorem would then imply that $S$ has a normal $p$-complement. This is impossible because $S$ is nonabelian simple and $p\mid |S|$. Thus
\[
 |N_G(P):C_G(P)|\ge2.
\]
Equation \eqref{eq:unionbound} yields 
\[
 |E_p(G)|\le \frac{|G|}{2},
\]
so $|\Delta_p(G)|\ge |G|/2$, again contradicting \eqref{eq:counter}. This completes the proof.
\end{proof}

\section{The $\rho^*$--$\sigma^*$ bound}

\begin{proof}[Proof of Theorem~\ref{thm:main}]
By Lemma~\ref{lem:pnil}, if $p\in\rhostar(G)$ then $G$ is not $p$-nilpotent. Hence Theorem~\ref{thm:density} gives
\[
 |\Delta_p(G)|\ge \frac{|G|}{2}
 \qquad\text{for every }p\in\rhostar(G).
\]
Since $G$ is nonnilpotent, not all sub-class sizes are $1$; equivalently, $\sigmastar(G)>0$ by \cite[Corollary~2.3]{QY}. Therefore
\begin{align*}
 \frac{|G|}{2}|\rhostar(G)|
 &\le \sum_{p\in\rhostar(G)}|\Delta_p(G)| \\
 &= \sum_{1\ne x\in G}\bigl|\pi(|x^{G^{*}}|)\bigr| \\
 &\le (|G|-1)\sigmastar(G) \\
 &< |G|\sigmastar(G).
\end{align*}
Dividing by $|G|/2$ gives
\[
 |\rhostar(G)|<2\sigmastar(G),
\]
as required.
\end{proof}

\section{Sharpness}

The coefficient $2$ in Theorem~\ref{thm:main} cannot be replaced by a smaller universal constant. We use the family from \cite[Example~1]{CD}.

\begin{proposition}\label{prop:sharp}
There is a family of finite supersolvable metabelian groups $G_n$ such that
\[
 \frac{|\rhostar(G_n)|}{\sigmastar(G_n)}\longrightarrow 2.
\]
\end{proposition}

\begin{proof}
Fix a prime $p$. Let $A$ be an elementary abelian $p$-group of rank $n$, and let $\mathcal M$ be the set of maximal subgroups of $A$. Thus
\[
 |\mathcal M|=\frac{p^n-1}{p-1}.
\]
For each $M\in\mathcal M$, choose distinct primes $q_M\equiv1\pmod p$; this is possible since there are infinitely many primes congruent to $1$ modulo $p$ (for example, by Dirichlet's theorem). Let $C_M$ be cyclic of order $q_M$, and let $A$ act nontrivially on $C_M$ with kernel $M$. Put
\[
 N=\prod_{M\in\mathcal M}C_M,
 \qquad G_n=N\rtimes A.
\]
These are the groups of \cite[Example~1]{CD}. Since both $N$ and $A$ are abelian, $G_n$ is metabelian; The normal factors $C_M$ and a cyclic series for $A$ give a normal series with cyclic factors, so $G_n$ is supersolvable.

Since $N$ is a normal $p$-complement of $G_n$, Lemma~\ref{lem:pnil} shows that $p\notin\rhostar(G_n)$. Fix $M\in\mathcal M$ and set $q=q_M$. We claim that
\[
 N_q(G_n)=NM,
\]
where $N_q(G_n)$ denotes the maximal normal $q$-nilpotent subgroup of $G_n$. Indeed, $NM\lhd G_n$, and $M$ centralizes $C_M$; hence
\[
 NM=C_M\times\left(\Bigl(\prod_{L\ne M}C_L\Bigr)\rtimes M\right).
\]
Thus $\left(\prod_{L\ne M}C_L\right)\rtimes M$ is a normal $q$-complement of $NM$, so $NM$ is $q$-nilpotent. Conversely, let $R\lhd G_n$ be $q$-nilpotent. If the image of $R$ in $A$ is not contained in $M$, choose $r\in R$ whose image lies outside $M$. Then $r$ acts nontrivially on $C_M$. Since $C_M$ has prime order, every nontrivial automorphism of $C_M$ has trivial fixed-point subgroup, and hence
\[
 [C_M,r]=C_M.
\]
Moreover $R\lhd G_n$ and $r\in R$, so $[C_M,r]\le R$. Thus $C_M\le R$.
Since $C_M$ is the Sylow $q$-subgroup of $R$, if $K$ is the normal $q$-complement of $R$, then $R=C_MK$ and $[C_M,K]=1$. Hence $C_M\le Z(R)$, contradicting the choice of $r$. Thus $R\le NM$, proving the claim.

Write $x=na$ with $n\in N$ and $a\in A$. By \eqref{eq:psolvable-exceptional}, applied with the prime $q_M$,
\[
 q_M\mid |x^{G_n^*}|
 \quad\Longleftrightarrow\quad
 x\notin NM
 \quad\Longleftrightarrow\quad
 a\notin M.
\]
Consequently
\[
 \pi(|x^{G_n^*}|)=\{q_M:M\in\mathcal M,\ a\notin M\}.
\]
Indeed, the only prime divisors of $|G_n|$ are $p$ and the primes $q_M$, while $p\notin\rhostar(G_n)$ by the normal $p$-complement $N$. Every $q_M$ occurs, and hence
\[
 |\rhostar(G_n)|=|\mathcal M|=\frac{p^n-1}{p-1}.
\]
If $a\ne1$, the maximal subgroups of $A$ containing $a$ are precisely the inverse images of the maximal subgroups of $A/\langle a\rangle$, so their number is
\[
 \frac{p^{n-1}-1}{p-1}.
\]
Thus exactly $p^{n-1}$ maximal subgroups of $A$ fail to contain $a$, and therefore
\[
 \sigmastar(G_n)=p^{n-1}.
\]
It follows that
\[
 \frac{|\rhostar(G_n)|}{\sigmastar(G_n)}
 =\frac{p^n-1}{(p-1)p^{n-1}}
 \longrightarrow \frac{p}{p-1}.
\]
Taking $p=2$ gives the required limit $2$.
\end{proof}

\section*{Acknowledgements}
This work was partially supported by a grant from the Simons Foundation (\#918096, to YY).

\section*{Disclosure Statement}
The authors declare that they have no competing interests and no conflicts of interest.

\section*{Data Availability Statement}
Data sharing is not applicable to this article, as no data sets were generated or analysed during the current study.

\end{document}